\documentclass[11pt]{amsart}

\scrollmode

\usepackage{times}
\usepackage{amsmath,graphicx}
\usepackage{latexsym}
\usepackage{amscd,amsthm,amssymb}
\usepackage[all]{xy}

\newtheorem{thm}{Theorem}
\newtheorem{prop}[thm]{Proposition}
\newtheorem{lem}[thm]{Lemma}

\newtheorem*{conj}{Conjecture}

\theoremstyle{definition}
\newtheorem{ex}[thm]{Example}
\newtheorem{rem}[thm]{Remark}

\title{Representing homology classes by symplectic surfaces}
\author{M.~J.~D.~Hamilton}
\address{      Institut f\"ur Geometrie und Topologie\\
               Universit\"at Stuttgart\\
               Pfaffenwaldring 57\\
               70569 Stuttgart\\
               Germany}
\email{mark.hamilton@math.lmu.de}
\date{\today}
\subjclass[2010]{Primary 57R17; Secondary 57N13, 57N35}
\keywords{4-manifold, symplectic, branched covering}

\begin{document}

\begin{abstract}We derive an obstruction to representing a homology class of a symplectic 4-manifold by an embedded, possibly disconnected, symplectic surface.
\end{abstract}

\maketitle

A natural question concerning symplectic 4-manifolds is the following: Given a closed symplectic 4-manifold $(M,\omega)$ and a homology class $B\in H_2(M;\mathbb{Z})$, determine whether there exists an embedded, possibly disconnected, closed symplectic surface representing the class $B$. This question has been studied by H.-V.~L\^e and T.-J.~Li \cite{Le, Li}. We always assume that the orientation of a symplectic surface is the one induced by the symplectic form. One necessary condition is then, of course, that the symplectic class $[\omega]$ evaluates positively on the class $B$, meaning that $\langle[\omega],B\rangle>0$. Among other things, it is shown in \cite{Li} that a class $B$ with $\langle[\omega],B\rangle>0$ in a symplectic 4-manifold is always represented by a symplectic {\em immersion} of a connected surface. It is also noted that an obstruction to representing a homology class $B$ by an embedded {\em connected} symplectic surface comes from the adjunction formula: The (even) integer 
\begin{equation*}
K_MB+B^2,
\end{equation*}
where $K_M$ denotes the canonical class of the symplectic 4-manifold $(M,\omega)$, has to be at least $-2$. This obstruction, however, disappears, if the number of components of the symplectic surface is allowed to grow large. Note that there are examples of classes in symplectic 4-manifolds which are represented by an embedded disconnected symplectic surface, but not by a connected symplectic surface: For example in the twofold blow-up $X\#2{\overline{\mathbb{CP}}{}^{2}}$ of any closed symplectic 4-manifold $X$ the sum of the classes of the exceptional spheres is not represented by a connected embedded symplectic surface according to the adjunction formula. It is the purpose of this article to derive an obstruction to representing a homology class by an embedded, possibly disconnected, symplectic surface.

In \cite{Li} it is also shown that for symplectic manifolds $M$ of dimension at least six, every class in $H_2(M;\mathbb{Z})$ on which the symplectic class evaluates positively is represented by a connected embedded symplectic surface. In \cite{Le} there is a conjecture which in the case of symplectic 4-manifolds $M$ says that if $\alpha$ is a class in $H_2(M;\mathbb{Z})$ on which the symplectic class evaluates positively, then there exists a positive integer $N$ depending on $\alpha$ such that $N\alpha$ is represented by an embedded, not necessarily connected, symplectic surface. In the examples at the end of this article we give counterexamples to this conjecture in the 4-dimensional case.

The non-existence of an embedded symplectic surface in the class $B$ has the following consequence for the Seiberg-Witten invariants, which we only state in the case $b_2^+>1$.
\begin{prop}\label{Taubes} Let $(M,\omega)$ be a closed symplectic 4-manifold with $b_2^+(M)>1$ and $B\neq 0$ an integral second homology class which cannot be represented by an embedded, possibly disconnected, symplectic surface. Then the Seiberg-Witten invariant of the $Spin^c$-structure 
\begin{equation*}
s_0\otimes PD(B)
\end{equation*}
is zero, where $s_0$ denotes the canonical $Spin^c$-structure with determinant line bundle $K_M^{-1}$ induced by a compatible almost complex structure.
\end{prop}
Here $PD$ denotes the Poincar\'e dual of a homology class. Note that the first Chern class of the $Spin^c$-structure $s_0\otimes PD(B)$ is equal to $-K_M+2PD(B)$. Proposition \ref{Taubes} is a consequence of a theorem of Taubes, relating classes with non-zero Seiberg-Witten invariants to embedded symplectic surfaces \cite{T2}.

In the following, let $(M,\omega)$ denote a closed symplectic 4-manifold and $\Sigma\subset M$ an embedded, possibly disconnected, closed symplectic surface representing a class $B\in H_2(M;\mathbb{Z})$. We always assume that the orientation of $M$ is given by the symplectic form ($\omega\wedge\omega>0$). If the class $B$ is divisible by an integer $d>1$, in the sense that there exists a class $A\in H_2(M;\mathbb{Z})$ such that $B=dA$, then there exists a $d$-fold cyclic ramified covering $\phi\colon\overline{M}\rightarrow M$, branched along $\Sigma$. The branched covering is again a closed symplectic 4-manifold. This is a well-known fact (the pullback of the symplectic form $\omega$ plus $t$ times a Thom form for the preimage $\overline{\Sigma}$ of the branch locus is for small positive $t$ a symplectic form on $\overline{M}$; see \cite {G, P} for a careful discussion). The invariants of $\overline{M}$ are given by the following formulas \cite[p.~243]{GS}, \cite{Hirz}:  

\begin{align*}
K_{\overline{M}}&=\phi^*(K_M+(d-1)PD(A))\\
K_{\overline{M}}^2&=d(K_M+(d-1)PD(A))^2\\
w_2(\overline{M})&=\phi^*(w_2(M)+(d-1)PD(A)_2)\\
\sigma(\overline{M})&=d\left(\sigma(M)-\frac{d^2-1}{3}A^2\right)\\
\end{align*} 
Here $PD(A)_2\in H^2(M;\mathbb{Z}_2)$ is the mod 2 reduction of $PD(A)$. The second equation follows from the first because the branched covering map has degree $d$.

Suppose that the branched covering $\overline{M}$ is symplectically minimal and not a ruled surface over a curve of genus greater than 1. Then theorems of C.~H.~Taubes and A.-K.~Liu \cite{Liu, T} imply that $K_{\overline{M}}^2\geq 0$. With the formula above, we get the following obstruction on the class $A$.
\begin{thm}\label{main thm} Let $(M,\omega)$ be a closed symplectic 4-manifold, $\Sigma\subset M$ an embedded, possibly disconnected, closed symplectic surface and $d>1$ an integer such that $dA=[\Sigma]$ for a class $A\in H_2(M;\mathbb{Z})$. Consider the $d$-fold cyclic branched cover $\overline{M}$, branched along $\Sigma$. If $\overline{M}$ is minimal and not a ruled surface over a curve of genus greater than 1, then 
\begin{equation*}
(K_M+(d-1)PD(A))^2\geq 0.
\end{equation*}
\end{thm}
It is therefore important to ensure that the branched covering $\overline{M}$ is minimal and not a ruled surface. First, we have the following lemma.
\begin{lem} Let $\phi\colon\overline{M}\rightarrow M$ be a cyclic $d$-fold branched covering of closed oriented 4-manifolds. Then $b_2^+(\overline{M})\geq b_2^+(M)$.
\end{lem}
\begin{proof} With our choice of orientations, the map $\phi\colon\overline{M}\rightarrow M$ has positive degree. By Poincar\'e duality, the induced map $\phi^*\colon H^*(M;\mathbb{R})\rightarrow H^*(\overline{M};\mathbb{R})$ is injective. It maps classes in the second cohomology of positive square to classes of positive square. This implies the claim.
\end{proof}
\begin{prop}\label{prop} In the notation of Theorem \ref{main thm}, each of the following two conditions implies that $\overline{M}$ is minimal and has $b_2^+(\overline{M})>1$ and hence is not a ruled surface:
\begin{enumerate}
\item If $d$ is odd assume that $M$ is spin and if $d$ is even assume that $PD(A)$ is characteristic. Also assume that $3\sigma(M)\neq (d^2-1)A^2$.
\item Assume that $b_2^+(M)\geq 2$ and there exists an integer $k\geq 2$ such that the class 
\begin{equation*}
K_M+(d-1)PD(A)
\end{equation*}
is divisible by $k$.
\end{enumerate}
\end{prop}
\begin{proof} Consider the $d$-fold branched covering $\overline{M}$, branched along $\Sigma$. The assumptions in case (a) imply that $\overline{M}$ is spin and that the signature $\sigma(\overline{M})$ is non-zero. According to a theorem of M.~Furuta \cite{F} we have $b_2^+(\overline{M})\geq 3$. Also the symplectic manifold $\overline{M}$ is minimal, because it is spin. In case (b) the lemma implies that $b_2^+(\overline{M})\geq 2$. In addition, the symplectic manifold $\overline{M}$ is minimal, because its canonical class is divisible by $k$ (a non-minimal symplectic 4-manifold $Y$ contains a symplectic sphere $S$ with $K_YS=-1$).
\end{proof}

\begin{ex}Consider $M=K3$. Then we have $K_M=0$. Let $d\geq 3$ be an integer and $A\in H_2(M;\mathbb{Z})$ a class with $A^2<0$. Theorem \ref{main thm} together with Proposition \ref{prop} part (b) imply that $dA$ is not represented by an embedded symplectic surface. Note that $K3$ contains indivisible classes of negative self-intersection which, for a suitable choice of symplectic structure, are represented by symplectic surfaces, for example symplectic $(-2)$-spheres. Let $A$ be the homology class of such a sphere and $\alpha=3A$. Then $\alpha$ is a counterexample to L\^e's Conjecture 1.4 in \cite{Le}.  
\end{ex}

\begin{ex}Let $X$ be a closed symplectic spin 4-manifold with $b_2^+>1$ and $M$ the blow-up $X\#{\overline{\mathbb{CP}}{}^{2}}$. Let $E$ denote the class of the exceptional sphere in $M$. We have $K_M=K_X+PD(E)$. For every positive even integer $d$ with $d^2>K_X^2$, the class $dE$ is not represented by a symplectic surface. Taking for example the blow-up of the $K3$ surface and $\alpha=2E$, we get another counterexample to L\^e's conjecture.  
\end{ex}

Note that with this method it is impossible to find a counterexample to L\^e's conjecture under the additional assumption that $\alpha^2>0$. 

In light of the second example, the following conjecture seems natural.
\begin{conj} Let $M$ be the blow-up $X\#{\overline{\mathbb{CP}}{}^{2}}$ of a closed symplectic 4-manifold $X$ and $E$ the class of the exceptional sphere. Then $dE$ is not represented by an embedded symplectic surface for all integers $d\geq 2$.
\end{conj}
This conjecture holds by a similar argument as above for $X$ the $K3$ surface and the 4-torus $T^4$. Moreover, using positivity of intersections, the conjecture holds in the complex category for the blow-up of a complex surface and embedded complex curves. In fact, in this category the result holds not only for the exceptional curve in a blow-up, but for multiples of the class of any connected embedded complex curve with negative self-intersection in a complex surface.

\begin{rem}Branched covering arguments have been used in the past to find lower bounds on the genus of a connected surface representing a divisible homology class in a closed 4-manifold, see \cite{B, HS, KM, R}. 
\end{rem}

\section*{Acknowledgements} \noindent I would like to thank D.~Kotschick for very helpful comments.

\bigskip
\bigskip

\end{document}